\documentclass[11pt]{amsart}
\usepackage{geometry}
\usepackage[all]{xy}
\usepackage{amsmath}
\usepackage{amssymb}
\usepackage{graphicx}
\usepackage{amsthm}
\usepackage[dvipdfmx]{overpic}
\DeclareMathOperator{\Ker}{Ker}
\DeclareMathOperator{\Ima}{Im}
\DeclareMathOperator{\Ck}{Coker}
\DeclareMathOperator{\Homom}{Hom}
\DeclareMathOperator{\Ext}{Ext}

\theoremstyle{plain}
 \newtheorem{theo}{Theorem}[section]
 \newtheorem{lemm}[theo]{Lemma}

\theoremstyle{definition}

\theoremstyle{remark}

\newcommand{\cal}{\mathcal}
\newcommand{\one}{\hspace{1mm}}

\newcommand{\om}{\omega}

\title{A Twisted First Homology group of the Goeritz group of $S^3$}
\author{Akira Kanada}
\begin{document}
\maketitle
\begin{abstract}
Given a genus-$g$ Heegaard splitting of a $3$-sphere,
the genus-$g$ Goeritz group $\cal{E}_g$ is defined to be the group of the isotopy classes of orientation preserving homeomorphism of the $3$-sphere that preserve the splitting. 
In this paper, we determine the twisted first (co)homology group of the genus-$2$ Goeritz group of $3$-sphere.
\end{abstract}

\section{Introduction}
\subsection*{Mapping class group}
Let $H_g$ be a 3-dimensional handlebody of genus $g$, and $\Sigma_g$ be the boundary surface $\partial H_g$.
We denote by $\cal{M}_g$ the mapping class group of $\Sigma_g$, the group of isotopy classes of orientation preserving homeomorphisms of $\Sigma_g$. Dehn \cite{Dehn} proved that $\cal{M}_g$ is generated by finitely many Dehn twists. Furthermore Lickorish \cite{Lic1,Lic2} proved that $3g-1$ Dehn twists generate $\cal{M}_g$.  Humphries \cite{Hum} found that $2g+1$ Dehn twists generate $\cal{M}_g$.\\
\one We denote by $\cal{H}_g$ the handlebody mapping class group, the subgroup of  mapping class group $\cal{M}_g$ of boundary surface $\partial H_g$ defined by isotopy classes of those orientation preserving homeomorphisms of  $\partial H_g$ which can be extended to homeomorphisms of $H_g$.  It turns out that $\cal{H}_g$ can be identified with the group of isotopy classes of orientation preserving homeomorphisms of $H_g$. A finite presentation of the handlebody mapping class group $\cal{H}_g$ was obtained by Wajnryb \cite{Waj}.
\subsection*{Goeritz group}
Let $H_g$ and $H_g^{*}$ be 3-dimensional handlebodies, and $M=H_g \cup H_g^{*}$ be a Heegaard splitting of a closed orientable $3$-manifold $M$. Let $\cal{M}_g$ be the mapping class group of the boundary surface $\partial H_g = \Sigma_g$. The group of  mapping classes $[f ] \in \cal{M}_g$
 such that there is an orientation preserving self-homeomorphism $F$ of $(M,H_g)$ satisfying $[F|_{\partial H_g}]=[f]$ is called the genus-$g$ Goeritz group of $M=H_g \cup H_g^{*}$. 
 When a manifold $M$ admits a unique Heegaard splitting of genus $g$ up to isotopy, we can define the genus-$g$ Goeritz group of  the manifold without mentioning a specific splitting. For example, the $3$-sphere, $S^1\times S^2$ and lens spaces are known to be such manifolds from \cite{W}, \cite{Bon1} and \cite{Bon2}.

In studying Goeritz groups, finding their generating sets or presentations has been an interesting problem. However the generating sets or the presentation of those groups have been obtained only for a few manifolds with their splittings of small genera. A finite presentation of the genus-2 Goeritz group of 3-sphere was obtained \cite{ER}.
In an arbitrary genus, first Powell \cite{Pow} and then Hirose \cite{Hi} claimed that they have found a finite generating set for the genus-$g$ Goeritz group of 3-sphere, though serious gaps in both arguments were found by Scharlemann. Establishing the existence of such generating sets appears to be an open problem.

In addition, finite presentations of the genus-2 Goeritz groups of each lens spaces $L(p,1)$ were obtained \cite{Cho1}, other lens spaces were obtained \cite{Cho2} and the genus-$2$ Heegaard splittings of non-prime $3$-manifolds were obtained \cite{Cho&Ko2}. Recently a finite presentation of  the genus-2 Goeritz group of $S^1\times S^2$ was obtained \cite{Cho&Ko}.  
\subsection*{Homology of mapping class group}
Computing homology of mapping class groups is interesting topic of studying mapping class groups. Harer \cite{Har} determined the second homology group of mapping class group $\cal{M}_g$:

\begin{align*} 
H_2(\cal{M}_g ; \mathbb{Z})\cong \mathbb{Z} \ \  \text{if}\ g\geq4.
\end{align*} 

In fact, Harer proved a more general theorem for surfaces with multiple boundary components and arbitrarily many punctures.

In twisted case, Morita \cite{MOR} determined the first homology group with coefficients in the first integral homology group of the surface:

\begin{align*} 
H_1(\cal{M}_g ; H_1(\Sigma_g))\cong \mathbb{Z}/(2g-2)\mathbb{Z} \ \  \text{if}\ g\geq2.
\end{align*} 

Recently Ishida and Sato \cite{To&Ishi} computed the twisted first homology groups of the handlebody mapping class group $\cal{H}_g$ with coefficients in the first integral homology group of the boundary surface $\Sigma_g$:

\begin{align*} 
H_1(\cal{H}_g ; H_1(\Sigma_g))\cong
\begin{cases}
\mathbb{Z}/(2g-2)\mathbb{Z} & \text{if}\ g\geq4,\\
\mathbb{Z}/2\mathbb{Z} \oplus \mathbb{Z}/4\mathbb{Z} & \text{if}\ g=3,\\
(\mathbb{Z}/2\mathbb{Z})^2 & \text{if}\ g=2.
\end{cases}
\end{align*}

\subsection*{Goeritz group of $S^3$}
 Let $H_g$ and $H_g^{*}$ be 3-dimensional handlebodies, and $S^3=H_g \cup H_g^{*}$ be the Heegaard splitting of the $3$-sphere $S^3$. Waldhausen \cite{W} proved that a genus-$g$ Heegaard splitting of $S^3$ is unique up to isotopy. Let $\cal{M}_g$ be the mapping class group of the boundary surface $\partial H_g = \Sigma_g$. The group of  mapping classes $[f ] \in \cal{M}_g$
 such that there is an orientation preserving self-homeomorphism $F$ of $(S^3,H_g)$ satisfying $[F|_{\partial H_g}]=[f]$ is denoted by $\cal{E}_g$. It is called the genus-$g$ Goeritz group of the $3$-sphere.
 
\subsection*{Twisted homology group of $\cal{E}_2$}
In this paper, we compute the twisted first homology group of $\cal{E}_2$ with coefficients in the first integral homology group of the Heegaard surface $\Sigma_2$. The following is the main theorem in this paper.
\begin{theo}
\begin{align*} 
H_1(\cal{E}_2; H_1(\Sigma_2) )\cong  (\mathbb{Z}/2\mathbb{Z})^2 .
\end{align*} 
\end{theo}
A finite presentation of the genus-2 Goeritz group of the 3-sphere was obtained from the works of \cite{ER}.
 For the higher genus Goeritz groups of the 3-sphere, it is conjectured that all of them are finitely presented however it is still known to be an open problem.\\

\newpage

Let $\Sigma_g$ be a compact connected orientable surface of genus $g\geq1$ and
$\alpha_1,\dots,\alpha_g,\beta_1,\dots,\beta_g$ be oriented simple closed curves as in Figure 1. 
We denote their homology classes in $H_1(\Sigma_g)$ by $x_1=[\alpha_1], x_2=[\alpha_2], \dots, x_g=[\alpha_g], y_1=[\beta_1], y_2=[\beta_2], \dots,y_g=[\beta_g]$. The basis $\{x_1,...,x_g,y_1,...,y_g\}$ of $H_A$ induces an isomorphism $H_A \simeq A^{2g}$. For $v\in A^{2g}$, we denote its projection to the $i$-th coordinate of $A^{2g}$ by $v_{i}$ for $i = 1,2,...,2g$.\\

\begin{center}
\begin{overpic}[width=11cm]{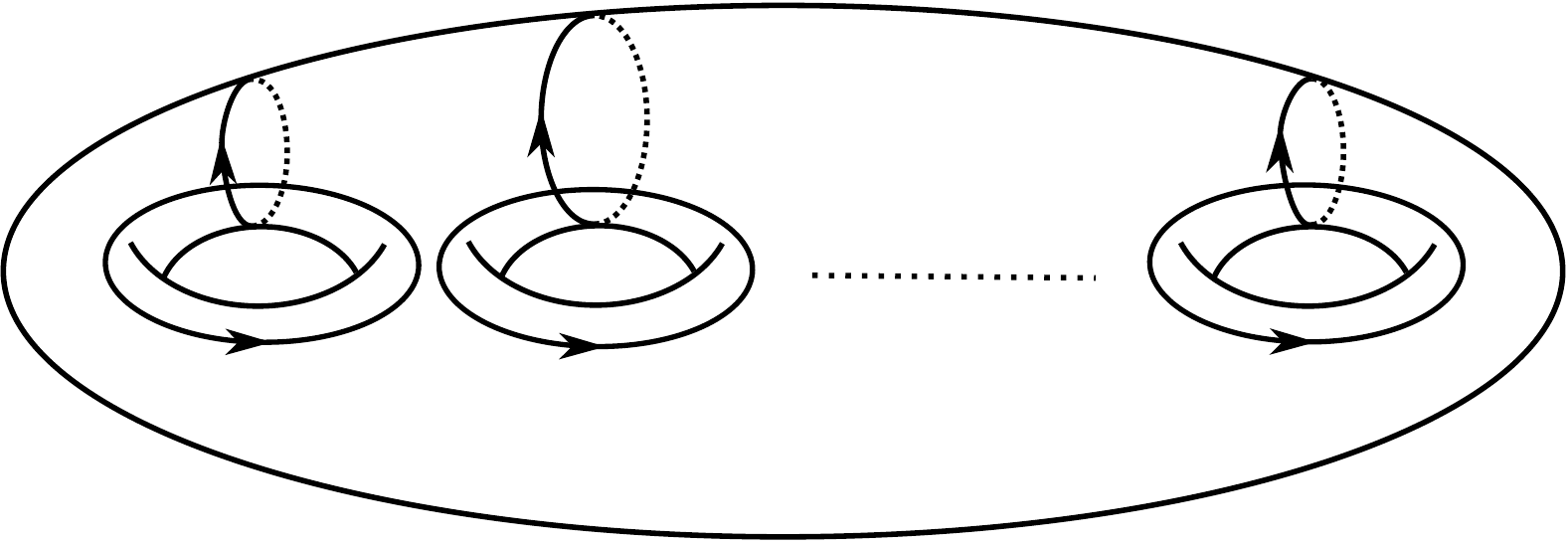}
\put(15,32){$\alpha_1$}
\put(37,36){$\alpha_2$}
\put(83,32){$\alpha_g$}
\put(15,17){$1$}
\put(37,17){$2$}
\put(83,17){$g$}
\put(15,8){$\beta_1$}
\put(37,8){$\beta_2$}
\put(83,8){$\beta_g$}
\put(0,-5){Figure 1 : Surface $\Sigma_g$ and simple closed curves $\alpha_1,\dots,\alpha_g,\beta_1,\dots,\beta_g$.}
\end{overpic}
\end{center}
\vspace{1cm}
Akbas gave following presentation for $\cal{E}_2$ in \cite{ER}.
\begin{theo}[\cite{ER}]
The group $\cal{E}_2$ has four generators $[\alpha],[\beta],[\gamma]$ and $[\delta]$, and the following  relations:
\begin{itemize}
\item[(P1)] $[\alpha]^2=[\beta]^2=[\delta]^3=[\alpha \gamma]^2=0$.
\item[(P2)] $[\alpha \delta \alpha]=[\delta]$ and $[\alpha \beta \alpha]=[\beta]$.
\item[(P3)] $[\gamma \beta \gamma]=[\alpha \beta]$ and $[\delta]=[\gamma \delta^2 \gamma]$.
\end{itemize}
\end{theo} 
\vspace{0.5cm}
\begin{center}
\begin{overpic}[width=11cm]{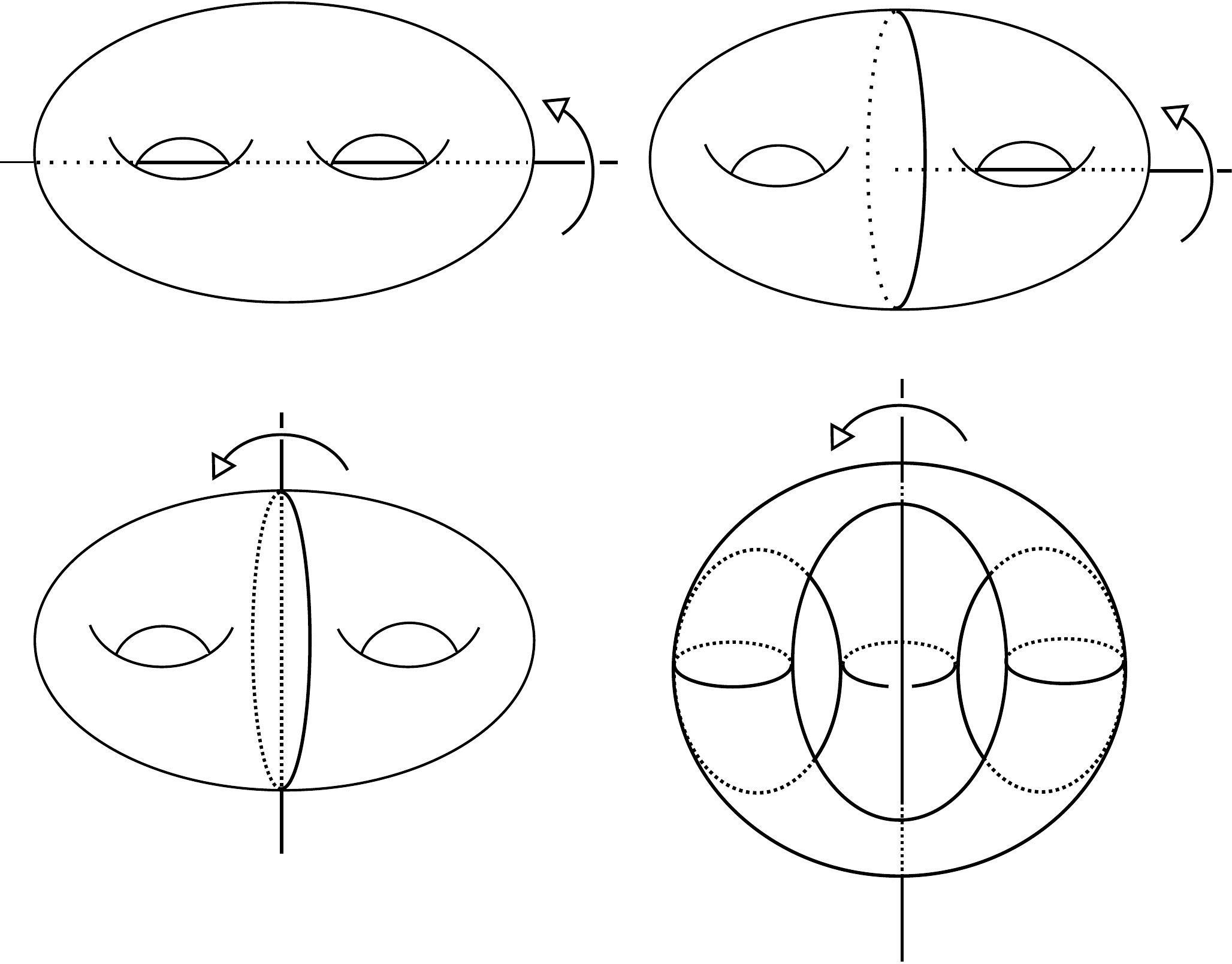}
\put(-5,75){(i)\ \ $\alpha$:}
\put(48,68){$\pi$}
\put(48,75){(ii)\ \ $\beta$:}
\put(98,68){$\pi$}
\put(-5,38){(iii)\ \ $\gamma$:}
\put(15,42){$\pi$}
\put(48,38){(iv)\ \ $\delta$:}
\put(65,45){$\frac{2\pi}{3}$}
\put(22,-5){Figure 2 : Generators of $\cal{E}_2$.}
\end{overpic}
\end{center}
\newpage
We define $\delta$ as follows. Consider the genus-two handlebody as a regular neighborhood of a sphere, centered at the origin, with three holes. The homeomorphism $\delta$ is a ${2\pi}/{3}$ rotation of the handlebody about the vertical $z$-axis. See Figure 2. Scharlemann \cite{Sch} showed that the group $\cal{E}_2$ is generated by isotopy classes $[\alpha],[\beta],[\gamma]$ and $[\delta]$.  Correspondence of homology classes of (iv) and the others are as follows:\\
\begin{center}
\begin{overpic}[width=7cm]{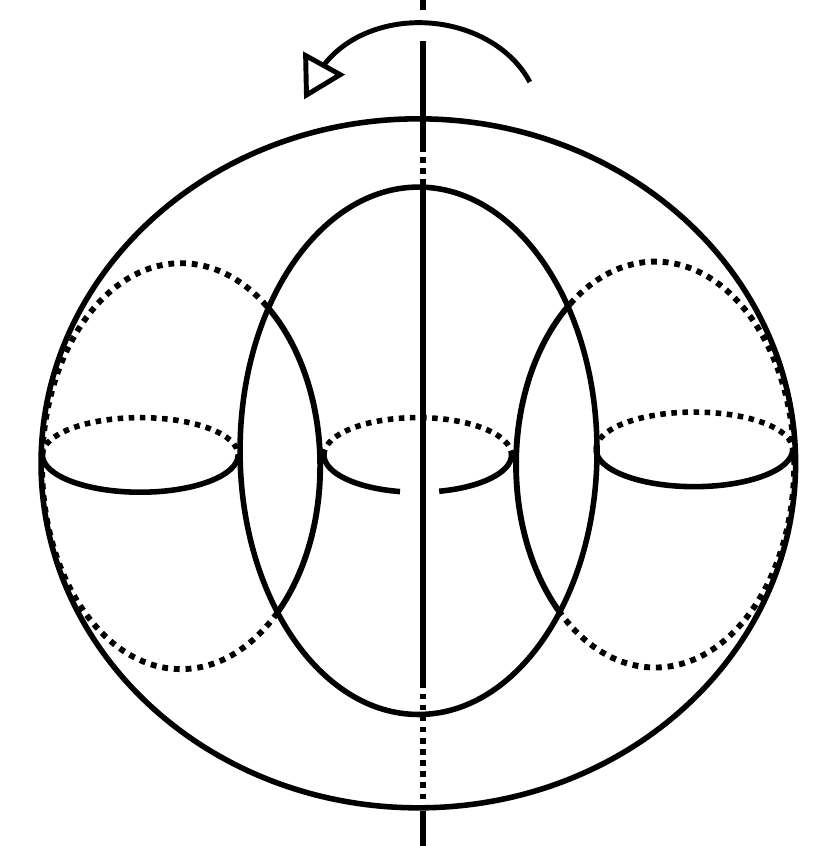}
\put(14,35){$x_1$}
\put(14,40.5){$<$}
\put(40,42.2){$\rotatebox{-20}{$>$}$}
\put(38.5,37){$x_1-x_2$}
\put(80,41.2){$>$}
\put(80,36){$x_2$}
\put(40,17.2){$\rotatebox{-30}{$>$}$}
\put(30,11){$y_1+y_2$}
\put(32.5,60){$\rotatebox{-64}{$<$}$}
\put(37.5,60){$y_1$}
\put(61,55){$\rotatebox{70}{$<$}$}
\put(57,60){$y_2$}
\end{overpic}
\end{center}


\section{Twisted cohomology group of   $\cal{E}_2$}
 
We denote by $A$ the ring $\mathbb{Z}$ or $\mathbb{Z}/n\mathbb{Z}$ for an integer $n\geq2$, and set $H_A = H_1(\Sigma_g;A)$. For a group $G$ and a left $G$-module $N$, a map $d:G \to N$ is called a \textit{crossed homomorphism} if it satisfies $d(gg')=d(g)+gd(g')$ for $g,g' \in G$.
Now let $Z^{1}(\cal{E}_g;H_A)$ be the set of all crossed homomorphisms $d:\cal{E}_g \to H_A$. Namely
\begin{align*}
Z^{1}(\cal{E}_g;H_A)=\{ d:\cal{E}_g \to H_A; d(\phi \psi)=d(\phi) + \phi d(\psi),\  \phi, \psi \in \cal{E}_g \}.
\end{align*}
Let $\pi : H_A \to Z^{1}(\cal{E}_g;H_A)$ be the homomorphism defined by 
\begin{align*}
\pi(u)(\phi)=\phi u - u
\end{align*}
for $u \in H_A$. Then as is well known we have
\begin{align*}
H^1(\cal{E}_g;H_A)= Z^{1}(\cal{E}_g;H_A)/\Ima\pi
\end{align*}
(cf. K.S. Brown \cite{Br}).

 
 We consider the case $g=2$. Then we have the homomorphism $\cal{E}_2 \to$ Aut$(H_1(\Sigma_2;\mathbb{Z}))$ induced by the action of the group $\cal{E}_2$ on $H_1(\Sigma_2;\mathbb{Z})$. The action of $\alpha,\beta,\gamma,$ and $\delta$ is as follows:
 \begin{itemize}
 \item[$\alpha_*$:]$\alpha_*(x_i)=-x_i$ and $\alpha_*(y_i)=-y_i$  ($i=1,2$).
 \item[$\beta_*$:]  $\beta_*(x_1)=x_1$, $\beta_*(x_2)=-x_2$, $\beta_*(y_1)=y_1$, $\beta_*(y_2)=-y_2$.
 \item[$\gamma_*$:] $\gamma_*(x_1)=-x_2$, $\gamma_*(x_2)=-x_1$,$\gamma_*(y_1)=-y_2$, $\gamma_*(y_2)=-y_1$.
 \item[$\delta_*$:] $\delta_*(x_1)=-x_1+x_2$, $\delta_*(x_2)=- x_1$, $\delta_*(y_1)=y_2$, $\delta_*(y_2)=-y_1-y_2$.
 \end{itemize}
For a group $G$ and a left $G$-module $N$, the \textit{coinvariant} $N_G$ is quotient module of $N$ by the subgroup $\{gn-n|g\in G, n\in N\}$.
\begin{lemm} 
\begin{align*} 
H_1(\Sigma_2)_{\mathcal{E}_2} =0.
\end{align*} 
\begin{proof}
Since we have $\alpha_* \delta_*^2(-x_2)=x_1-x_2$ and $\gamma_*(x_1)=-x_2$, we obtain $x_1=x_2=0\in H_1(\Sigma_2)_{\mathcal{E}_2}$.
And we have $\gamma_* \delta_*^2(y_1)=y_1+y_2$ and $\gamma_*(y_1)=-y_2$.
Hence we also obtain $y_1=y_2=0\in H_1(\Sigma_2)_{\mathcal{E}_2}$.
\end{proof}
\end{lemm}

\begin{lemm}Let $G_i$, $H_i$ and $K$ be G-modules $(i=1,2,3)$, and let
\begin{eqnarray*} 
&\cdots \to G_3\to G_2\to K \to G_1\to 0 \ \text{and}\\
&0 \to H_1\to K \to H_2 \to H_3 \to \cdots
\end{eqnarray*} 
be exact sequences. If $G_2 \to K \to H_2$ is an isomorphism, then we have $H_1 \to K \to G_1$ is isomorphism.
\end{lemm}

\begin{proof}
 Now we have the following diagram. Set $\Phi=g_1\circ f_1$ and $\Psi=f_2\circ g_2$.
\\
\\
$$
\xymatrix @R=15pt{
& & & & & 0 \ar[ld] \\
 & & & & H_1 \ar[ld]_{f_1} \ar@{.>} [d]\\
\cdots \ar[r] & G_3 \ar[r]^{g_3} & G_2 \ar@{.>}[d] \ar[r]^{g_2} & K \ar[ld]^{f_2}  \ar[r]_{g_1} & G_1 \ar[r] & 0 \\
  & & H_2 \ar[ld]^{f_3}\\
  & H_3 \ar[ld]\\
\rotatebox{70}{$\ddots$}
 }
$$
\\
\\
Note that $g_2$ is injective and $f_2$ is surjective if $G_2 \to K \to H_2$ is an isomorphism. Thus we have $g_3=0$ and $f_3=0$. 
Since $\Psi$ is isomorphism, we have $(\Psi^{-1} \circ f_2)\circ g_2=id_{G_{2}}$ and $f_2\circ(g_2 \circ \Psi^{-1})=id_{H_{2}}$. Hence those exact sequences are split and we have 
\begin{align*} 
&H_1 \oplus H_2 \overset{\simeq}{\longrightarrow} K: \quad (h_1,h_2) \mapsto f_1(h_1)+g_2\circ \Psi^{-1}(h_2),\\
&K  \overset{\simeq}{\longrightarrow} G_1 \oplus G_2 :\quad k \mapsto (g_{1}(k),\Psi^{-1}\circ f_2(k)).
\end{align*} 
A composition map $H_1 \oplus H_2 \to K \to G_1 \oplus G_2$ is 
\begin{eqnarray*} 
(h_1,h_2)&\mapsto& (g_1(f_1(h_1)+g_2 \circ \Psi^{-1}(h_2)),\Psi^{-1}\circ f_2(f_1(h_1)+g_2 \circ \Psi^{-1}(h_2)))\\
&=&(g_1\circ f_1(h_1)+g_1\circ g_2 \circ \Psi^{-1}(h_2),\Psi^{-1}\circ f_2 \circ f_1(h_1)+\Psi^{-1}\circ f_2 \circ g_2 \circ \Psi^{-1}(h_2))\\
&=&(\Phi(h_1),\Psi^{-1}(h_2)).
\end{eqnarray*} 
Hence $\Psi$ is an isomorphism.
\end{proof}
\newpage
\begin{lemm}In the case $g=2$, we have  
\begin{align*} 
H^1(\cal{E}_2;H_A)&\cong \{d \in Z^{1}(\cal{E}_2;H_A); \\& d([\delta])_1 - d([\alpha])_1 = d([\gamma])_2 - d([\beta])_2 = d([\gamma])_3 - d([\alpha])_3 = d([\beta])_4 - d([\delta])_4 = 0 \}.
\end{align*} 
\end{lemm}
\begin{proof}
\one Let $f: Z^{1}(\cal{E}_2;H_A) \to A^{4}$ be a homomorphism defined by
\begin{align*}
f(d)=(d([\delta])_1 - d([\alpha])_1 , d([\gamma])_2 - d([\beta])_2 , d([\gamma])_3 - d([\alpha])_3 , d([\beta])_4 - d([\delta])_4 ).
\end{align*}
Since we have
\begin{align*}
\alpha v -v= &(-2v_1 , -2v_2 ,-2v_3 ,-2v_4),& \\
\beta v - v= &(0 , -2v_2 ,0 ,-2v_4),&\\
\gamma v - v=&(-v_1- v_2 , -v_1- v_2 ,-v_3-v_4 ,- v_3-v_4),&\\
\delta v - v =&(-2v_1+ v_2, -v_1- v_2, -v_3+ v_4, -v_3- 2v_4),&
\end{align*}
the composition map $f\circ \pi \one: \one H_A \to A^4$ is written as 
\[
f\circ \pi(v)=(v_2 , -v_1+v_2 , v_3-v_4 , -v_3) \hspace{5mm} 
\]
for $v \in H_A$. This map is an isomorphism. We have the following diagram.
$$
\xymatrix{
 & & & \Ker f \ar[ld] \ar@{.>} [d]  & 0 \ar[l] \\
 0 \ar[r] & H_A \ar@{.>}[d]_{\cong} \ar[r]^{\pi \hspace{5mm} \ } & Z^{1}(\cal{E}_2;H_A) \ar[ld]^{f}  \ar[r] & \Ck \pi \ar[r] & 0 \\
 0  & A^4 \ar[l]\\
 }
$$
By Lemma $2.2$, we have
\[
H^1(\cal{E}_2;H_A)= Z^{1}(\cal{E}_2;H_A)/\Ima\pi \cong \Ker f.
\]
\end{proof}

The group $\cal{E}_2$ is generated by $[\alpha],[\beta],[\gamma],$ and $[\delta]$. Therefore, all crossed homomorphisms $d : \cal{E}_2 \to H_A$ are determined by the values  $d([\alpha]),d([\beta]),d([\gamma])$ and $d([\delta]) $. If $d \in Z^{1}(\cal{E}_2,H_A)$, we can set
 
\begin{align*}
d([\alpha])=\om_{11}x_1 + \om_{12}x_2 + \om_{13}y_1 + \om_{14}y_2, \\
d([\beta])=\om_{21}x_1 + \om_{22}x_2 + \om_{23}y_1 + \om_{24}y_2,\\
d([\gamma])=\om_{31}x_1 + \om_{32}x_2 + \om_{33}y_1 + \om_{34}y_2,\\
d([\delta])=\om_{41}x_1 + \om_{42}x_2 + \om_{43}y_1 + \om_{44}y_2.
\end{align*}
 Then we have    
\begin{align*}
H^1(\cal{E}_2;H_A)\cong \{d \in Z^{1}(\cal{E}_2;H_A)\  ; \  \om_{41} - \om_{11} = \om_{32} - \om_{22} =\om_{33} - \om_{13} = \om_{24} - \om_{44} = 0 \}.
\end{align*}
\section{relations of $\om_{ij}$}
In this section we shall consider the case $g=2$. We denote $d([\phi])$ for $\phi \in \cal{E}_2$ simply by $d(\phi)$.
\begin{lemm}
We have relations:
\begin{align*}
\begin{cases}
\om_{11} + \om_{12} = \om_{31} + \om_{32}, &(1a) \\
\om_{13} + \om_{14} = \om_{33} + \om_{34}. &(1b)
\end{cases}
\end{align*}
\end{lemm}

\begin{proof}
By the relations $(\alpha \gamma)^2=0$ in (P1), we have $d((\alpha \gamma)^2)=0$. The equation
\begin{eqnarray*}
d((\alpha\gamma)^2)
&=& d(\alpha\gamma)+\alpha\gamma d(\alpha\gamma)\\
&=& d(\alpha)+\alpha d(\alpha)+\alpha\gamma d(\alpha) + \alpha\gamma\alpha d(\gamma)\\
&=& \om_{11}x_1 + \om_{12}x_2 + \om_{13}y_1 + \om_{14}y_2 +( -\om_{31}x_1 - \om_{32}x_2 - \om_{33}y_1 - \om_{34}y_2 )\\
&& \hspace{-2.5mm} +\om_{12}x_1 + \om_{11}x_2 + \om_{14}y_1 + \om_{13}y_2 + ( -\om_{32}x_1 - \om_{31}x_2 - \om_{34}y_1 - \om_{33}y_2) \\
&=&(\om_{11} + \om_{12} - \om_{31} - \om_{32})x_1 + (\om_{11} + \om_{12} - \om_{31} - \om_{32})x_2 \\
&&\hspace{-2.5mm} +(\om_{13} + \om_{14} - \om_{33} - \om_{34})y_1 +(\om_{13} + \om_{14} - \om_{33} - \om_{34})y_2\\
&=&0
\end{eqnarray*}
holds.
Hence, we obtain $(1a)$ and $(1b)$.
\end{proof}

\begin{lemm}
We have relations:
\begin{align*}
\begin{cases}
2\om_{11} + \one \one  \om_{12} = 2\om_{41}, & (2a) \\
\one -\om_{11} + \om_{12} = 2\om_{42},  &(2b) \\
\one\one \om_{13} +\one \one  \om_{14} = 2\om_{43}, &(2c) \\
-\om_{13} + 2\om_{14} = 2\om_{44}, & (2d)
\end{cases}
\end{align*}
\end{lemm}

\begin{proof}
By the relations $\alpha \delta \alpha=\delta$ in (P2), we have $d(\alpha \delta \alpha)=d(\delta)$ and
\begin{eqnarray*}
d(\alpha \delta \alpha)
&=&d(\alpha)+\alpha d(\delta) + \alpha \delta d(\alpha) \\
&=&\om_{11}x_1 + \om_{12}x_2 + \om_{13}y_1 +\om_{14}y_2 +(-\om_{41}x_1 - \om_{42}x_2 - \om_{43}y_1 - \om_{44}y_2) \\
&& \hspace{-2.5mm} +\om_{11}(x_1-x_2) + \om_{12}x_1- \om_{13}y_2 + \om_{14}(y_1+y_2) \\
&=& (2\om_{11} - \om_{41} + \om_{12})x_1 +(\om_{12} - \om_{42} - \om_{11})x_2 \\
&& \hspace{-2.5mm} +(\om_{13} - \om_{43} + \om_{14})y_2 +(2\om_{14} - \om_{44} - \om_{13})y_2, \\
d(\delta)&=&\om_{41}x_1 + \om_{42}x_2 + \om_{43}y_1 + \om_{44}y_2.
\end{eqnarray*}
Comparing $d(\alpha \delta \alpha)$ and $d(\delta)$, we obtain $(2a)-(2d)$.
\end{proof}

\begin{lemm} We have relations:
\begin{align*}
\begin{cases}
2\om_{21} = 2\om_{23} = 0 , & (3a) \\
2\om_{12} = 2\om_{22}, & (3b) \\
2\om_{14} = 2\om_{24}. & (3c)
\end{cases}
\end{align*}
\end{lemm}

\begin{proof}
By the relation $\alpha \beta \alpha=\beta$ in (P2), we have $d(\alpha \beta \alpha) = d(\beta)$ and
\begin{eqnarray*}
d(\alpha \beta \alpha) 
&=& d(\alpha) + \alpha d(\beta \alpha) \\
&=& d(\alpha) + \alpha d(\beta) + \alpha \beta d(\alpha) \\
&=& \om_{11}x_1 + \om_{12}x_2 + \om_{13}y_1 + \om_{14}y_2 \\
&&\hspace{-2.5mm} -\om_{21}x_1 - \om_{22}x_2 - \om_{23}y_1 - \om_{24}y_2 \\
&&\hspace{-2.5mm} -\om_{11}x_1 + \om_{12}x_2 - \om_{13}y_1 + \om_{14}y_2 \\
&=&\hspace{-2.5mm} -\om_{21}x_1 + (2\om_{12} - \om_{22})x_2 -\om_{23}y_1 +(2\om_{14} - \om_{24})y_2, \\
d(\beta)&=&\om_{21}x_1 + \om_{22}x_2 + \om_{23}y_2 + \om_{24}y_2.
\end{eqnarray*}
Comparing  $d(\alpha \beta \alpha)$ and $d(\beta)$, we obtain $(3a)-(3c)$.
\end{proof}

\begin{lemm} We have relations:
\begin{align*}
\begin{cases}
\om_{31} - \om_{32} - \om_{22} = \om_{11} + \om_{21}, &  (4a) \\
\om_{32} + \om_{31} - \om_{21} = \om_{12} - \om_{22}, &  (4b)\\
\om_{33} - \om_{34} - \om_{24} = \om_{13} + \om_{23},  & (4c) \\
\om_{34} + \om_{33} - \om_{23} = \om_{14} - \om_{24}.  & (4d) 
\end{cases}
\end{align*}
\end{lemm}

\begin{proof}
By the relation $\gamma \beta \gamma=\alpha \beta$ in (P3), we have $d(\gamma \beta \gamma)=d(\alpha \gamma)$ and
\begin{eqnarray*}
d(\gamma \beta \gamma) 
&=&d(\gamma) + \gamma d(\beta \gamma) \\
&=&d(\gamma) + \gamma d(\beta) + \gamma \beta d(\gamma) \\
&=&\om_{31}x_1 + \om_{32}x_2 + \om_{33}y_1 + \om_{34}y_2 \\
&&\hspace{-2.5mm} -\om_{21}x_2 - \om_{22}x_1 - \om_{23}y_2 - \om_{24}y_1 \\
&&\hspace{-2.5mm} +\om_{31}x_2 - \om_{32}x_1 + \om_{33}y_2 - \om_{34}y_1 \\
&=& (\om_{31} - \om_{32} - \om_{22})x_1 + (\om_{31} + \om_{32} - \om_{21})x_2 \\
&&\hspace{-2.5mm} + (\om_{33} - \om_{34} - \om_{24})y_1 + (\om_{33} + \om_{34} - \om_{23})y_2, \\
d(\alpha \beta)
&=& d(\alpha) + \alpha d(\beta) \\
&=& (\om_{11} + \om_{21})x_1 + (\om_{12} - \om_{22})x_2 + (\om_{13} - \om_{23})y_1 + (\om_{14} + \om_{24})y_2.
\end{eqnarray*}
Comparing  $d(\gamma \beta \gamma)$ and $d(\alpha \beta)$, we obtain $(4a)-(4d)$.
\end{proof}

\begin{lemm} We have relations:
\begin{align*}
\begin{cases}
2\om_{31}+\om_{32} = 2\om_{41} + \om_{42}, & (5a) \\
\hspace{1.2cm} \om_{33} = \one \one \om_{43}. & (5b)
\end{cases}
\end{align*}
\end{lemm}

\begin{proof}
By the relation $\gamma \delta^2 \gamma=\delta$ in (P3), we have $d(\gamma \delta^2 \gamma)=d(\delta)$ and
\begin{eqnarray*}
d(\gamma \delta^2 \gamma) 
&=&d(\gamma) + \gamma d(\delta^2 \gamma) \\
&=&d(\gamma) + \gamma d(\delta) + \gamma \delta d(\delta) + \gamma \delta^2 d(\gamma) \\
&=&\om_{31}x_1 + \om_{32}x_2 + \om_{33}y_1 + \om_{34}y_2 \\
&&\hspace{-2.5mm} - \om_{42}x_1 -\om_{41}x_2  - \om_{44}y_1- \om_{43}y_2  \\
&&\hspace{-2.5mm}  -\om_{41}x_1 +(\om_{41}+ \om_{42})x_2 + (-\om_{43}+\om_{44})y_1 + \om_{44}y_2 \\
&&\hspace{-2.5mm}  +(\om_{31}+\om_{32})x_1-\om_{32}x_2 +\om_{33}y_1 + (\om_{33}-\om_{34})y_2 \\
&=& (2\om_{31} - \om_{41} - \om_{42}+\om_{32})x_1 + \om_{42}x_2\\
&&\hspace{-2.5mm} + (2\om_{33} - \om_{43})y_1 + (\om_{33} + \om_{44} - \om_{43})y_2, \\
d(\delta)
&=&\om_{41}x_1 + \om_{42}x_2 + \om_{43}y_2 + \om_{44}y_2.
\end{eqnarray*}
Comparing  $d(\gamma \delta^2 \gamma)$ and $d(\delta)$, we obtain $(5a)$ and $(5b)$.
\end{proof}

By Lemma 3.1, 3.2, 3.3, 3.4 and 3.5, we obtain the following equations. 
\\
\\
\hspace{5mm}
$
\begin{cases}
\om_{11} + \om_{12} = \om_{31} + \om_{32}, &(1a) \\
\om_{13} + \om_{14} = \om_{33} + \om_{34}, &(1b)
\end{cases}
$
\hspace{5mm}
$
\begin{cases}
2\om_{11} + \one \one  \om_{12} = 2\om_{41}, & (2a) \\
\one -\om_{11} + \om_{12} = 2\om_{42},  &(2b) \\
\one\one \om_{13} +\one \one  \om_{14} = 2\om_{43}, &(2c) \\
-\om_{13} +2\om_{14} = 2\om_{44}, & (2d)
\end{cases}
$
\\
\\
\\
\hspace{5mm}
$
\begin{cases}
2\om_{21} = 2\om_{23} = 0, &(3a) \\
2\om_{12} = 2\om_{22}, &(3b) \\
2\om_{14} = 2\om_{24}, & (3c)
\end{cases}
$
$
\hspace{5mm}\hspace{5mm}\hspace{5mm}
\begin{cases}
\om_{31} - \om_{32} - \om_{22} = \om_{11} + \om_{21}, &  (4a) \\
\om_{31} + \om_{32} - \om_{21} = \om_{12} - \om_{22}, & (4b) \\
\om_{33} - \om_{34} - \om_{24} = \om_{13} + \om_{23},  &(4b) \\
\om_{34} + \om_{33} - \om_{23} = \om_{14} - \om_{24}, & (4b)
\end{cases}
$
\\
\\
\\
\hspace{5mm}
$
\begin{cases}
2\om_{31}+\om_{32} = 2\om_{41} + \om_{42}, & (5a) \\
\hspace{1.2cm} \om_{33} = \one \one \om_{43}. & (5b)
\end{cases}
$
\\
\\
\\
\section{Calculation of cohomology}
In this section we prove that $H^1(\mathcal{E}_2;H_A)\cong \Homom( (\mathbb{Z}/2\mathbb{Z})^2,A)$. The universal coefficient theorem implies $H_1(\cal{E}_2; H_1(\Sigma_2) )\cong  (\mathbb{Z}/2\mathbb{Z})^2$.
To determine the twisted first cohomology group of $\cal{E}_2$, we solve equations $(1a)-(5b)$ and the condition $d \in \Ker f$, i.e.
\begin{align}
\om_{41} - \om_{11} = \om_{32} - \om_{22} =\om_{33} - \om_{13} = \om_{24} - \om_{44} = 0.&\tag{$*$}
\end{align}
\begin{lemm}We have a relation of $\Ker f$ :
\begin{align*}
\om_{12}= 0.
\end{align*}
\end{lemm}

\begin{proof}
Using $(2a)$ and $\om_{11}=\om_{41}$ by $(*)$, we obtain
\begin{align*}
\om_{12} = 0.
\end{align*}
\end{proof}

\begin{lemm} The elements $\om_{21}$, $\om_{22}$, $\om_{31}$, $\om_{32}$ and $\om_{42}$ have order $2$ and 
\begin{align*} 
\om_{22} =\om_{31}=\om_{32} = -\om_{21}=-\om_{42} .
\end{align*}
\end{lemm}

\begin{proof}
By (2a) and (2b), we have $\om_{11} + 2\om_{12} = 2\om_{41} + 2\om_{42}$. By $(5a)$ we have  $2\om_{41} + 2\om_{42} = 2\om_{31} + \om_{32}+\om_{42}$. Using these two equations and $(1a)$, we have
$
\om_{12}=\om_{31}+\om_{42}. 
$ 
Since $\om_{12}=0$, the equation 
\begin{align*} 
\om_{31}+\om_{42}=0 \one\hspace{5mm}\hspace{5mm} (4.2.1)
\end{align*} 
holds. Using $(2b)$ and Lemma 4.1, we obtain
\begin{align*}
\om_{11}+2\om_{42}=0. \hspace{9mm} (4.2.2)
\end{align*}
The equation
\begin{eqnarray*}
\om_{31}+\om_{32} &\overset{(1a)}{=}&
\om_{11}+ \om_{12}\\ &\overset{(Lem 4.1)}{=}&
\om_{11} \\ &\overset{(4.2.2)}{=}& 
-2\om_{42}\\ &\overset{(4.2.1)}{=}&
2\om_{31}
\end{eqnarray*}
holds. So we obtain
\begin{eqnarray*}
\om_{31}- \om_{32}=0.\hspace{15mm}(4.2.3)
\end{eqnarray*}
We have
\begin{eqnarray*} 
\om_{21} &\overset{(4a)}{=}& \om_{31}-\om_{32}-\om_{11}-\om_{22}\\
&\overset{(4.2.3)}{=}& -\om_{11} -\om_{22}\\
&\overset{(**)}{=}& -2\om_{22} -\om_{22}\\
&=& -3\om_{22}\\
&\overset{(3b)}{=}& -\om_{22} -2\om_{12}\\
&\overset{(Lem 4.1)}{=}& -\om_{22}.
\end{eqnarray*} 
Here, $(**)$ is the equation $\om_{11}=2\om_{22}$. This equation is obtained as follows:
\begin{eqnarray*}
\om_{11}\overset{(4.2.2)}{=}-\om_{42}\overset{(4.2.1)}{=}2\om_{31}\overset{(4.2.3)}{=}2\om_{32}\overset{(*)}{=}2\om_{22}.
\end{eqnarray*}
By $(3b)$ and Lemma4.1, we have $-2\om_{22}=0$. Hence we obtain
\begin{eqnarray*}
2\om_{21}=-2\om_{22}=0.
\end{eqnarray*}
\end{proof}
From the results of Lemma 4.2 and (4.2.2), we have 
\begin{align*}
\om_{11}\overset{(*)}{=}\om_{41}=0.
\end{align*}

\begin{lemm} We have a relation:
\begin{align*}
\om_{13}=\om_{14}=\om_{33}=\om_{34}=\om_{43}=0.
\end{align*}
\end{lemm}

\begin{proof}
By $(5b)$ and $(*)$, we obtain 
\begin{align*}
\om_{13}=\om_{33}=\om_{43}=\om_{14}.
\end{align*}
Using this equation and $(1b)$, we have $\om_{14}=\om_{34}$. 
So the equation
\begin{eqnarray*}
\hspace{30mm} \om_{13}=\om_{33}=\om_{43}=\om_{14}=\om_{34}  \hspace{10mm}(4.3.1)
\end{eqnarray*}
holds. 
The equation 
\begin{eqnarray*}
2\om_{44} &\overset{(2d)}{=}& 2\om_{14}-\om_{13}\\
&\overset{(4.3.1)}{=}& \om_{14}
\end{eqnarray*}
holds. Hence we have 
 \begin{align*}
4\om_{44}=2\om_{14} \overset{(3c)}{=} 2\om_{24} \overset{(*)}{=}2\om_{44}.
\end{align*}
So we obtain 
\begin{align*}
2\om_{44}=0.
\end{align*}
Since the equation $2\om_{44}=\om_{14}$ holds, we complete the proof of Lemma 4.3.
\end{proof}
\begin{lemm} The elements $\om_{23}$, $\om_{24}$ and $\om_{44}$ have order $2$ and are equal to each other:
\begin{align*}
\om_{23}=\om_{24}=\om_{44},\quad 2\om_{23} =0.
\end{align*}
\end{lemm}
\begin{proof}
By (3a), the element $\om_{23}$ have order $2$. 
Since the equation
\begin{eqnarray*}
\om_{24} &\overset{(4b)}{=}& \om_{14} -\om_{33}-\om_{34} + \om_{23} \\
&\overset{(Lem 4.3)}{=}& \om_{23} 
\end{eqnarray*}
holds, we obtain
\begin{align*}
\om_{23}=\om_{24}\overset{(*)}{=}\om_{44}.
\end{align*}
\end{proof}

\subsection*{Proof of Theorem 1.1.}
From Lemmas 3,7, 4.1, 4.2, 4.3 and 4.4, we have 
\begin{eqnarray*}
d(\alpha)&=&0,\\
d(\beta)&=&\om_{22}(-x_1+x_2) +\om_{23}(y_1+y_2),\\
d(\gamma)&=&\om_{22}(x_1+x_2),\\
d(\delta)&=&\om_{42}x_2 + \om_{23}y_2,
\end{eqnarray*}
where $-\om_{22}=\om_{42}$ and $2\om_{22}=2\om_{23}=0$. Hence it follows
\begin{align*}
H^1(\cal{E}_2;H_{A})\cong \Ker f \cong\{(\om_{22}, \om_{23})\in A^2 ; 2\om_{22}=2\om_{23}=0\}.
\end{align*}
So we obtain
\begin{align*}
H_1(\cal{E}_2;H_1(\Sigma_2))\cong  (\mathbb{Z}/2\mathbb{Z})^2.
\end{align*}
This isomorphism follows from the short exact sequence
\begin{align*}
0 \to \Ext(H_0(\cal{E}_2;H_1(\Sigma_2)),A) \to H^1(\cal{E}_2;H_A) \to \Homom(H_1(\cal{E}_2;H_1(\Sigma_2)),A) \to 0
\end{align*}
since we have $H_0(\cal{E}_2;H_1(\Sigma_2))=H_1(\Sigma_2)_{\cal{E}_{2}}=0$ by Lemma 2.1.
we complete the proof of Theorem 1.1.


\vspace{0.5cm}

\address(Akira Kanada){
Department of Mathematics, Tokyo Institute of Technology, Oh-okayama, Meguro, Tokyo 152-8551, Japan}

\email{
E-mail address: kanada.a.aa@m.titech.ac.jp
}

\begin{thebibliography}{99}
 \bibitem{Waj}Wajnryb, Bronistaw. "Mapping class group of a handlebody." Fundamenta Mathematicae 158.3 (1998): 195-228.
 \bibitem{Dehn}M.Dehn. "Die gruppe der abdildungsklassen." Acta Math. 69 (1938), 135-206.
 \bibitem{ER} Akbas, Erol. "A presentation for the automorphisms of the 3-sphere that preserve a genus two Heegaard splitting." Pacific Journal of Mathematics 236.2 (2008): 201-222.
 \bibitem{Bon1}Bonahon, Francis. "Diff\'eotopies des espaces lenticulaires." Topology 22.3 (1983): 305-314.
 \bibitem{Bon2}Bonahon, Francis, and Jean-Pierre Otal. "Scindements de Heegaard des espaces lenticulaires." Annales 
 \bibitem{W} Waldhausen, Friedhelm. "Heegaard-Zerlegungen der 3-sph\"are." Topology 7.2 (1968): 195-203.
  \bibitem{To&Ishi}Tomohiko Ishida and Masatoshi Sato. "A twisted first homology group of the handlebody mapping class group" arXiv preprint arXiv:1502.07048v1 (2015).
 \bibitem{Pow}Powell, Jerome. "Homeomorphisms of $S^3$ leaving a Heegaard surface invariant." Transactions of the American Mathematical Society 257.1 (1980): 193-216.
 \bibitem{Har} Harer, John. "The second homology group of the mapping class group of an orientable surface." Inventiones Mathematicae 72.2 (1983): 221-239.
 \bibitem{Br} Brown, Kenneth S. Cohomology of groups. Vol. 87. Springer Science \& Business Media, 2012.
 \bibitem{Sch}Scharlemann, Martin. "Automorphisms of the 3-sphere that preserve a genus two Heegaard splitting." arXiv preprint math/0307231 (2003).
 \bibitem{Cho1}Cho, Sangbum. "Genus-two Goeritz groups of lens spaces." Pacific Journal of Mathematics 265.1 (2013): 1-16.
 \bibitem{Cho&Ko}Cho, Sangbum, and Yuya Koda. "The genus two Goeritz group of $ S^ 2\ times \ S^ 1$." arXiv preprint arXiv:1303.7145 (2013).
\bibitem{Cho&Ko2}Cho, Sangbum, and Yuya Koda. "The Goeritz groups of Heegaard splittings for 3-manifolds."
 \bibitem{Cho2}Cho, Sangbum, and Yuya Koda. "Connected primitive disk complexes and genus two Goeritz groups of lens spaces." International Mathematics Research Notices (2016): rnv399.
 \bibitem{Hum}Humphries, Stephen P. "Generators for the mapping class group." Topology of low-dimensional manifolds. Springer Berlin Heidelberg, 1979. 44-47.
 \bibitem{Hi} Hirose, Susumu. "Homeomorphisms of a 3-dimensional handlebody standardly embedded in S3." Proceedings of Knot 96 (1997): 493-513.
 \bibitem{MOR} Morita, Shigeyuki. "Families of Jacobian manifolds and characteristic classes of surface bundles. I." Annales de l'institut Fourier. Vol. 39. No. 3. 1989.
Scientifiques de l'\'Ecole Normale Superieure. Vol. 16. No. 3. 1983.
 \bibitem{Lic1}Lickorish, WB Raymond. "A representation of orientable combinatorial 3-manifolds." Annals of Mathematics (1962): 531-540.
 \bibitem{Lic2}Lickorish, William Bernard Raymond. "A finite set of generators for the homeotopy group of a 2-manifold." Mathematical Proceedings of the Cambridge Philosophical Society. Vol. 60. No. 04. Cambridge University Press, 1964.
\end{thebibliography}
\end{document}